\documentclass[reqno]{amsart}

\usepackage{amsmath}
\usepackage{amssymb}
\usepackage[bookmarks]{hyperref}
\usepackage{enumerate}

\newtheorem{theorem}{Theorem}[section]
\newtheorem{lemma}[theorem]{Lemma}
\newtheorem{corollary}[theorem]{Corollary}
\newtheorem{proposition}[theorem]{Proposition}

\newtheorem*{theoremA}{Theorem A}
\newtheorem*{theoremB}{Theorem B}

\theoremstyle{definition}
\newtheorem{remark}[theorem]{Remark}

\numberwithin{equation}{section}

\newcommand{\D}{\mathbb{D}}

\begin{document}

\title[Adjoints of composition operators]{Adjoints of linear fractional composition operators on weighted Hardy spaces}

\author{{\v Z}eljko {\v C}u{\v c}kovi{\'c}}
\address{Department of Mathematics and Statistics, Mail Stop 942, University of Toledo, Toledo, OH 43606}
\email{zeljko.cuckovic@utoledo.edu}

\author{Trieu Le}
\address{Department of Mathematics and Statistics, Mail Stop 942, University of Toledo, Toledo, OH 43606}
\email{trieu.le2@utoledo.edu}

\subjclass[2010]{Primary 47B33}

\keywords{Composition operator; adjoint; weighted Hardy space}

\begin{abstract} 
It is well known that on the Hardy space $H^2(\mathbb{D})$ or weighted Bergman space $A^2_{\alpha}(\mathbb{D})$ over the unit disk, the adjoint of a linear fractional composition operator equals the product of a composition operator and two Toeplitz operators. On $S^2(\mathbb{D})$, the space of analytic functions on the disk whose first derivatives belong to $H^2(\mathbb{D})$, Heller showed that a similar formula holds modulo the ideal of compact operators. In this paper we investigate what the situation is like on other weighted Hardy spaces.
\end{abstract}

\maketitle

\section{Introduction}\label{S:intro} 

Let $\D$ denote the open unit disk in the complex plane. Let $\varphi:\D\rightarrow\D$ be an analytic map. The composition operator $C_{\varphi}$ is defined by $C_{\varphi}f = f\circ\varphi$, where $f$ is an analytic function on $\D$. Composition operators have been studied extensively on Hilbert spaces of analytic functions such as the Hardy space $H^2$, the weighted Bergman spaces $A^2_{\alpha}$ ($\alpha>-1$) and the Dirichlet space $\mathcal{D}$, just to name few. The reader is referred to the excellent books \cite{CowenCRCP1995} and \cite{Shapiro1993} for more details. Of particular interest was finding the formula for the adjoint $C_{\varphi}^{*}$ on these spaces. Cowen \cite{CowenIEOT1988} found the formula for $C_{\varphi}^{*}$ on $H^2$ for the case $\varphi$ is a linear fractional self-map of $\D$ (we shall call such $C_{\varphi}$ a linear fractional composition operator). Cowen showed that if $\varphi(z) = (az+b)/(cz+d)$ is a linear fractional mapping of $\D$ into itself then
\begin{equation}\label{Eqn:CowenFormula}
C_{\varphi}^{*}=M_gC_{\sigma}M_h^{*},
\end{equation}
where $\sigma(z)=(\bar{a}z-\bar{c})/(-\bar{b}z+\bar{d})$ is the Kre\u{\i}n adjoint of $\varphi$ and $M_g$ and $M_h$ are multiplication operators with symbols $g(z)=(-\bar{b}z+\bar{d})^{-1}$ and $h(z)=cz+d$. Cowen's formula was later extended by Hurst \cite{HurstArchivM1997} to weighted Bergman spaces $A^2_{\alpha}$ with $\alpha>-1$. Such formulas initiated more studies of the adjoint of linear fractional composition operators on different spaces of analytic functions and on $H^2$ for  general rational symbols. See \cite{Gallardo-GutierrezMAnn2003, CowenJFA2006,MartinJFA2006,HammondJMAA2008,BourdonJFA2008} and the references therein.

Recently, Heller \cite{HellerJMAA2012} investigated the adjoint of $C_{\varphi}$ acting on the space $S^2(\D)$, which consists of analytic functions on $\D$ whose first derivative belongs to $H^2$. Let $\mathcal{K}$ denote the ideal of compact operators on $S^2(\D)$. Heller obtained the following results.
\begin{theoremA}
Let $\varphi(z)=az/(cz+d)$ be a holomorphic self-map of the disk and consider $C_{\varphi}$ acting on $S^2(\D)$. Then $$C_{\varphi}^{*} = M_{G}^{*}C_{\sigma} \text{ mod } \mathcal{K},$$
where $G(z)=(-c/a)z+1$ and $\sigma(z)=(\bar{a}/\bar{d})z-\bar{c}/\bar{d}$ is the Kre\u{\i}n adjoint of $\varphi$.
\end{theoremA}

\begin{theoremB}
Let $\varphi(z)=\lambda(z+u)/(1+\bar{u}z)$, $|\lambda|=1$, $|u|<1$, be an automorphism of the disk and consider $C_{\varphi}$ acting on $S^2(\D)$. Then $$C_{\varphi}^{*} = M_{G}^{*}C_{\varphi^{-1}}M_{1/H} \text{ mod } \mathcal{K},$$
where $G(z)=-\overline{\lambda u}\,z+1$ and $H(z)=\bar{u}z+1$.
\end{theoremB}
For a general linear fractional self-map $\varphi$, a formula for $C_{\varphi}^{*}$ modulo the compact operators can be obtained by combining the above two results. Certain simplification of the above formulas was also presented in \cite{HellerJMAA2012}. It is curious to us that Heller's formulas are not of the same form as Cowen's formula \eqref{Eqn:CowenFormula}: the order of the multiplication operators is different. 
The purpose of the paper is to investigate the adjoints of linear fractional composition operators in a more general setting. We then explain how to recover Heller's formulas from our results.

All of the spaces mentioned above belong to the class of weighted Hardy spaces $H^2(\beta)$, where $\beta=\{\beta(n)\}_{n\geq 1}^{\infty}$ is a sequence of positive numbers. These spaces are Hilbert spaces of analytic functions on the unit disk in which the monomials $\{z^n:n\geq 0\}$ form an orthogonal basis with $\|z^n\|=\beta(n)$. We shall show that it is possible to obtain Cowen's formula modulo compact operators not only on $S^2(\D)$ but also on a wide subclass of weighted Hardy spaces $H^2(\beta)$. Our strategy involves the family of weighted Bergman spaces $A^2_{\alpha}$ ($\alpha\in\mathbb{R}$) studied by Zhao and Zhu \cite{ZhaoMSMF2008}. We use the exact formulas for the reproducing kernels of $A^2_{\alpha}$ to obtain Cowen's type formula for $C_{\varphi}^{*}$ on these spaces first. We then extend our formulas to $H^2(\beta)$ for appropriate weight sequences whose term $\beta_n$ behaves asymptotically as $\|z^n\|_{\alpha}$.

\section{Adjoint formulas on $A^2_{\alpha}$}

In this section we study the adjoint of composition operators acting on weighted Bergman spaces $A^2_{\alpha}$ for $a\in\mathbb{R}$. The standard weighted Bergman spaces are defined for measures $dA_{\alpha}(z) = (1-|z|^2)^{\alpha}dA(z)$ with $\alpha>-1$. Zhao and Zhu \cite{ZhaoMSMF2008} extended this definition to the case where $\alpha$ is any real number. For any $\alpha\in\mathbb{R}$, the space $A^2_{\alpha}$ consists of holomorphic functions $f$ on $\D$ with the property that there exists an integer $k\geq 0$ with $\alpha+2k>-1$ such that $(1-|z|^2)^{k}f^{(k)}(z)$ belongs to $L^2(\D,dA_{\alpha})$, or equivalently, $f^{(k)}$ belongs to $A^2_{\alpha+2k}$. It is well know that this definition is consistent with the traditional definition for $\alpha>-1$. The reader is referred to \cite{ZhaoMSMF2008} for a detailed study of $A^2_{\alpha}$. Note that any function that is analytic on an open neighborhood of the closed unit disk belongs to $A^2_{\alpha}$ for all $\alpha$.

In \cite[Section 11]{ZhaoMSMF2008}, it was shown that each $A^2_{\alpha}$ is a reproducing kernel Hilbert space. When equipped with an appropriate inner product, the kernel of $A^2_{\alpha}$ can be computed explicitly. Depending on the value of $\alpha$, we obtain three types of kernels. For each type, we show that the operator $$C_{\varphi}^{*}-M_gC_{\sigma}M^{*}_{h}$$ is either zero or has finite rank, where $g$ and $h$ are certain analytic functions associated with $\varphi$.

For $\alpha+2>0$, the kernel is
\begin{align}
\label{Eqn:Kernel_typeA}
K_{\alpha}(z,w) = \dfrac{1}{(1-z\bar{w})^{\alpha+2}},
\end{align}
and $$\|z^m\|_{\alpha}=\sqrt{\frac{m!\,\Gamma(\alpha+2)}{\Gamma(m+\alpha+2)}},\quad m=0,1,2\ldots,$$ which behaves asymptotically as $m^{-(\alpha+1)/2}$ by Stirling's formula.

When $\alpha+2$ is negative and non-integer such that $-N<\alpha+2<-N+1$ for some positive integer $N$, the kernel takes the form
\begin{align}
\label{Eqn:Kernel_typeB}
K_{\alpha}(z,w) = \dfrac{(-1)^{N}}{(1-z\bar{w})^{\alpha+2}} +Q(z\bar{w}),
\end{align}
where $Q$ is an analytic polynomial of degree $N$. In this case, for $m>N$,
$$\|z^m\|_{\alpha} = \sqrt{(-1)^N\dfrac{m!\, \Gamma(\alpha+2)}{\Gamma(m+\alpha+2)}},$$ which also behaves asymptotically as $m^{-(\alpha+1)/2}$.

In the case $\alpha+2 = -N$, where $N$ is a non-negative integer, the kernel has the form
\begin{align}
\label{Eqn:Kernel_typeC}
K_{\alpha}(z,w)=\big(\bar{w}z-1\big)^{N}\log\big(\dfrac{1}{1-\bar{w}z}\big)+Q(\bar{w}z),
\end{align}
where $Q$ is an analytic polynomial of degree $N$. For $m>N$, we have
$$\|z^m\|_{\alpha} = \sqrt{\dfrac{1}{A_m}},$$
where $A_m$ is the coefficient of $z^m$ in the Taylor expansion
\begin{align*}
(z-1)^N\log\dfrac{1}{1-z} & = \sum_{k=0}^{\infty}A_k z^k.
\end{align*}
The argument in the paragraph preceding \cite[Theorem 44]{ZhaoMSMF2008} shows that $\|z^m\|_{\alpha}$ behaves asymptotically as $m^{(N+1)/2} = m^{-(\alpha+1)/2}$ as well.

\begin{remark}
\label{R:asymptoticBehavior}
For any real number $\alpha$, we see that $\|z^m\|_{\alpha}$ behaves asymptotically as $m^{-(\alpha+1)/2}$ when $m\rightarrow\infty$.
\end{remark}

For the Hardy and weighted Bergman spaces (which may be identified as $A^2_{\alpha}$ with $\alpha\geq -1$), it is well known that their multiplier spaces are exactly the same as $H^{\infty}(\D)$ and any composition operator induced by a holomorphic self-map of $\D$ is bounded. However, such results do not hold for other values of $\alpha$. On the other hand, it turns out, as we shall show below, that all multiplication and composition operators discussed in this paper are bounded on all $A^2_{\alpha}$. 

For two positive quantities $A$ and $B$, we write $A\lesssim B$ if there exists a constant $c>0$ independent of the variables under consideration such that $A\leq cB$. We write $A\approx B$ if $A\lesssim B$ and $B\lesssim A$.

Let $m\geq 0$ be an integer. Recall \cite[Theorem 13]{ZhaoMSMF2008} that for any real number $\alpha$, a function $f$ belongs to $A^2_{\alpha}$ if and only if the $m$th derivative $f^{(m)}$ belongs to $A^2_{\alpha+2m}$ and 
\begin{equation}
\label{Eqn:equivNorm}
\|f\|_{\alpha}\approx \|f^{(m)}\|_{\alpha+2m}.
\end{equation}
Also, if $\alpha_1<\alpha_2$ then 
\begin{equation}
\label{Eqn:inequalNorm}
\|\cdot\|_{\alpha_2}\lesssim\|\cdot\|_{\alpha_1}.
\end{equation}

\begin{lemma}
\label{L:multipliers_A2_alpha}
Let $\alpha$ be a real number and $m$ be a positive integer such that $\alpha+2m> -1$. There exists a positive constant $C$ such that if $u$ is a function holomorphic on an open neighborhood of the closed unit disk, then $M_{u}$ is a bounded operator on $A^2_{\alpha}$ and
\begin{align}
\label{Eqn:multipliers_A2_alpha}
 \|M_{u}\|\leq C\max\{\|u^{(j)}\|_{L^{\infty}(\D)}: 0\leq j\leq m\}.
\end{align}
\end{lemma}

\begin{proof}
To simplify the notation, we put $$\|u\|_{m,\infty} = \max\{\|u^{(j)}\|_{L^{\infty}(\D)}: 0\leq j\leq m\}.$$
For any $f\in A^2_{\alpha}$, using \eqref{Eqn:equivNorm}, we compute
\begin{align*}
\|u f\|_{\alpha} \approx \|(u f)^{(m)}\|_{\alpha+2m} & = \Big\|\sum_{j=0}^{m}\binom{m}{j}u^{(m-j)}f^{(j)}\Big\|_{\alpha+2m}\\
& \leq \sum_{j=0}^{m}\binom{m}{j}\|u^{(m-j)}\|_{L^{\infty}(\D)}\|f^{(j)}\|_{\alpha+2m}\\
& \leq \|u\|_{m,\infty}\sum_{j=0}^{m}\binom{m}{j}\|f^{(j)}\|_{\alpha+2m}.
\end{align*}
Moreover, for any $0\leq j\leq m$, by \eqref{Eqn:equivNorm} and \eqref{Eqn:inequalNorm}, we have
\begin{align*}
\|f^{(j)}\|_{\alpha+2m} \approx \|f\|_{\alpha+2m-2j} \lesssim \|f\|_{\alpha}.
\end{align*} 
Consequently,
\begin{align*}
\|u f\|_{\alpha} & \lesssim \|u\|_{m,\infty}\|f\|_{\alpha}\sum_{j=0}^{m}\binom{m}{j} = 2^{m}\|u\|_{m,\infty}\|f\|_{\alpha}.
\end{align*}
This implies \eqref{Eqn:multipliers_A2_alpha} with a constant $C$ independent of $u$. 
\end{proof}

\begin{lemma}
\label{L:boundedCOs}
Let $\varphi$ be a holomorphic self-map of $\D$ such that $\varphi$ extends to a holomorphic function on an open neighborhood of the closed unit disk. Then $C_{\varphi}$ is a bounded operator on $A^2_{\alpha}$ for any real number $\alpha$.
\end{lemma}

\begin{proof}
Fix any real number $\gamma>-1$. We shall prove that $C_{\varphi}$ is bounded on $A^2_{-2k+\gamma}$ for all integers $k\geq 0$ by induction on $k$. This immediately yields the conclusion of the lemma.

Since $\gamma>-1$, $A^2_{\gamma}$ is the weighted Bergman space with weight $(1-|z|^2)^{\gamma}$. It is well known that $C_{\varphi}$ is bounded on $A^2_{\gamma}$, which proves our claim for the case $k=0$. Now assume that $C_{\varphi}$ is bounded on $A^2_{-2k+\gamma}$ for some integer $k\geq 0$. We would like to show that $C_{\varphi}$ is bounded on $A^2_{-2k-2+\gamma}$. Since $C_{\varphi}$ is a closed operator, it suffices to show that for any $h$ in $A^2_{-2k-2+\gamma}$, the composition $h\circ\varphi$ belongs to $A^2_{-2k-2+\gamma}$ as well. This, in turn, is equivalent to the requirement that $(h\circ\varphi)'$ belongs to $A^2_{-2k+\gamma}$. We have $(h\circ\varphi)' = (h'\circ\varphi)\cdot\varphi'$. Since $h$ is in $A^2_{-2k-2+\gamma}$, the  derivative $h'$ belongs to $A^2_{-2k+\gamma}$. By the induction hypothesis, $h'\circ\varphi = C_{\varphi}h'$ belongs to $A^2_{-2k+\gamma}$ as well. On the other hand, by our assumption about $\varphi$, Lemma \ref{L:multipliers_A2_alpha} shows that multiplication by $\varphi'$ is a bounded operator on $A^2_{-2k+\gamma}$. Consequently, $(h'\circ\varphi)\cdot\varphi'$ is an element of $A^2_{-2k+\gamma}$, which is what we wish to show.
\end{proof}

As in Heller's work, our adjoint formula for $C_{\varphi}$ holds modulo finite rank or compact operators. We first recall a description of finite rank operators on Hilbert spaces, see, for example, \cite[Exercise II.4.8]{ConwaySpringer1990}.

Let $\mathcal{H}$ be a Hilbert space. For non-zero vectors $u,v\in\mathcal{H}$, we use $u\otimes v$ to denote the rank one operator $(u\otimes v)(h) = \langle h,v\rangle u$ for $h\in\mathcal{H}$. 

\begin{lemma}
\label{L:finiteRankHilbert}
A bounded linear operator $A:\mathcal{H}\rightarrow\mathcal{H}$ has rank at most $m$ if and only if there exist $f_1,\ldots, f_m$ and $g_1,\ldots, g_m$ belonging to $\mathcal{H}$ such that
$$A = f_1\otimes g_1 + \cdots + f_m\otimes g_m.$$
\end{lemma}

When $\mathcal{H}$ is a reproducing kernel Hilbert space of analytic function, Lemma \ref{L:finiteRankHilbert} takes a different form which will be useful for us. This result is probably well known but we provide a proof for the reason of completeness.

\begin{lemma}
\label{L:finiteRankOperators}
Let $\mathcal{H}$ be a Hilbert space of analytic functions on the unit disk with reproducing kernel $K$. Let $\mathcal{X}$ be the set of functions on $\mathbb{D}\times\mathbb{D}$ of the form
\begin{equation}
\label{Eqn:quasi_Functions}
f_1(z)\overline{g_1(w)}+\cdots+f_m(z)\overline{g_m(w)},
\end{equation}
where $f_1,\ldots, f_m$ and $g_1,\ldots, g_m$ belong to $\mathcal{H}$ and $m$ is a positive integer. Then a bounded linear operator $A:\mathcal{H}\longrightarrow\mathcal{H}$ has finite rank if and only if the function $(z,w)\mapsto\langle AK_w,K_z\rangle$ belongs to $\mathcal{X}$. Here $K_{w}(z) = K(z,w)$ for $z,w\in\D$.
\end{lemma}

\begin{proof}
By Lemma \ref{L:finiteRankHilbert}, a bounded linear operator $A:\mathcal{H}\longrightarrow\mathcal{H}$ has finite rank if and only if there exist a positive integer $m$ and functions $f_1,\ldots, f_m$ and $g_1,\ldots, g_m$ belonging to $\mathcal{H}$ such that
$$A = f_1\otimes g_1 + \cdots + f_m\otimes g_m.$$
For $1\leq i,j\leq m$, we have
$$\langle (f_j\otimes g_j)K_w,K_z\rangle = \langle K_w,g_j\rangle\langle f_j,K_z\rangle = f_j(z)\overline{g_j(w)}.$$
The conclusion of the lemma now follows from the density of the linear span of $\{K_w:w\in\D\}$.
\end{proof}

Suppose $\varphi(z)=(az+b)/(cz+d)$ is a linear fractional self-map of the unit disk. Let $\sigma(z)=(\bar{a}z-\bar{c})/(-\bar{b}z+\bar{d})$ be the Kre\u{\i}n adjoint of $\varphi$. It is known that $\sigma$ is also a self-map of $\mathbb{D}$. Let $\eta(z)=(cz+d)^{-1}$ and $\mu(z)=-\bar{b}z+\bar{d}$. Then $\eta$ and $\mu$ are bounded analytic functions on a neighborhood of the closed unit disk and
\begin{align*}
1-\overline{\varphi(w)}z & = \mu(z)\big(1-\bar{w}\sigma(z)\big)\overline{\eta(w)}.
\end{align*}
Consequently, by choosing appropriate branches of the logarithms, we have
\begin{align}
\label{Eqn:log}
\log\big(1-\overline{\varphi(w)}z\big) & = \log(\mu(z)) + \log\big(1-\bar{w}\sigma(z)\big) + \log(\overline{\eta(w)}).
\end{align}
Therefore, for any real number $\gamma$,
\begin{align}
\label{Eqn:exponent}
\Big(1-\overline{\varphi(w)}z\Big)^{\gamma} & = \mu(z)^{\gamma}\Big(1-\bar{w}\sigma(z)\Big)^{\gamma}\overline{\eta(w)^{\gamma}}
\end{align}
for $z, w$ in $\D$.

We are now in a position to discuss the adjoints of composition operators induced by linear fractional maps. In the following theorem, we consider the first two types of kernels.
\begin{theorem}
\label{T:AdjointFormula_a}
Let $\alpha$ be a real number such that $\alpha+2$ is not zero nor a negative integer. Let $\varphi(z)=(az+b)/(cz+d)$ be a linear fractional self-map of the unit disk and $\sigma$ be its Kre\u{\i}n adjoint. Let $g(z) = (-\bar{b}z+\bar{d})^{-\alpha-2}$ and $h(z)=(cz+d)^{\alpha+2}$ for $z\in\D$. Then $C_{\varphi}^{*}-M_gC_{\sigma}M^{*}_{h}$ has finite rank on $A^2_{\alpha}$. In the case $\alpha+2>0$, we actually have the identity $C_{\varphi}^{*} = M_{g}C_{\sigma}M^{*}_{h}$.
\end{theorem}

\begin{remark}
Lemmas \ref{L:multipliers_A2_alpha} and \ref{L:boundedCOs} show that the operators $C_{\varphi}$, $C_{\sigma}$, $M_{g}$ and $M_{h}$ are all bounded on $A^2_{\alpha}$.
\end{remark}

\begin{proof}[Proof of Theorem \ref{T:AdjointFormula_a}] 
As we mentioned before, the case $\alpha>-1$ was considered by Hurst \cite{HurstArchivM1997}. His proof works also for $-2<\alpha\leq -1$ since the kernels have the same form. Here we only need to investigate the case $-N < \alpha+2 < -N+1$ for some positive integer $N$. To simplify the notation, let $\gamma=-(\alpha+2)$. We then rewrite the kernel as $K(z,w) = (-1)^N(1-\bar{w}z)^{\gamma}+Q(\bar{w}z)$ for $z,w\in\mathbb{D}$. Set $K_w(z) = K(z,w)$ for $z,w\in\D$. We shall make use of the following identities, which are well known,
\begin{align*}
M^{*}_{h}K_w = \overline{h(w)}K_w,\quad M^{*}_{g}K_z = \overline{g(z)}K_z, \quad C_{\varphi}^{*}K_w = K_{\varphi(w)}. 
\end{align*}
We now compute
\begin{align*}
\langle(C_{\varphi}^{*} - M_{g}C_{\sigma}M^{*}_{h})K_{w},K_z\rangle & = K(z,\varphi(w)) - g(z)K(\sigma(z),w)\overline{h(w)}\\
& = (-1)^N(1-\overline{\varphi(w)}z)^{\gamma} + Q(\overline{\varphi(w)}z)\\
&\quad - g(z)\Big((-1)^N(1-\bar{w}\sigma(z))^{\gamma}+Q(\bar{w}\sigma(z))\Big)\overline{h(w)}\\
&= Q(\overline{\varphi(w)}z) - g(z)Q(\bar{w}\sigma(z))\overline{h(w)} \quad \text{(using \eqref{Eqn:exponent}).}
\end{align*}
Since $g$ and $h$ are analytic on a neighborhood of the closed unit disk and $Q$ is a polynomial, the last function has the form \eqref{Eqn:quasi_Functions}. Consequently, Lemma \ref{L:finiteRankOperators} shows that $C_{\varphi}^{*} - M_{g}C_{\sigma}M^{*}_{h}$ has finite rank. 
\end{proof}

The following theorem considers the third type of kernel.

\begin{theorem}
\label{T:AdjointFormula_b}
Let $\alpha$ be a real number such that $\alpha+2$ is zero or a negative integer. Let $\varphi(z)=(az+b)/(cz+d)$ be a linear fractional self-map of the unit disk and $\sigma$ be its Kre\u{\i}n adjoint. Let $g(z) = (-\bar{b}z+\bar{d})^{-\alpha-2}$ and $h(z)=(cz+d)^{\alpha+2}$ for $z\in\D$. Then $C_{\varphi}^{*}-M_gC_{\sigma}M^{*}_{h}$ has finite rank on $A^2_{\alpha}$.
\end{theorem}

\begin{proof}
Let $N=-\alpha-2$. Then $N$ is a nonnegative integer. Recall that the kernel in this case has the form
$$K(z,w)=\big(\bar{w}z-1\big)^{N}\log\big(\dfrac{1}{1-\bar{w}z}\big)+Q(\bar{w}z)$$ for $z,w\in\mathbb{D}$,  where $Q$ is an analytic polynomial. We compute
\begin{align*}
&\langle(C_{\varphi}^{*} - M_{g}C_{\sigma}M^{*}_{h})K_{w},K_z\rangle\\
&\ = K\big(z,\varphi(w)\big) - g(z)K\big(\sigma(z),w\big)\overline{h(w)}\\
&\ = -\big(\overline{\varphi(w)}z-1\big)^{N}\log\big(1-\overline{\varphi(w)}z\big) + g(z)(\bar{w}\sigma(z)-1)^{N}\log(1-\bar{w}\sigma(z))\overline{h(w)}\\
&\qquad\qquad + Q(\overline{\varphi(w)}z) - g(z)Q(\bar{w}\sigma(z))\overline{h(w)}.
\end{align*}
Since $g(z)(\bar{w}\sigma(z)-1)^{N}\overline{h(w)}=(\overline{\varphi(w)}z-1)^{N}$, using \eqref{Eqn:log}, we simplify the first two terms in the last expression as
\begin{align*}
&\big(\overline{\varphi(w)}z-1\big)^{N}\Big(-\log(1-\overline{\varphi(w)}z)+\log(1-\bar{w}\sigma(z))\Big) & \\
&\qquad\qquad =-\big(\overline{\varphi(w)}z-1\big)^{N}\Big(\log(\mu(z))+\log(\overline{\eta(w)})\Big),
\end{align*}
where $\eta(z)=(cz+d)^{-1}$ and $\mu(z)=-\bar{b}z+\bar{d}$ for $z\in\D$. Consequently,
\begin{align*}
\langle(C_{\varphi}^{*} - M_{g}C_{\sigma}M^{*}_{h})K_{w},K_z\rangle
& = -\big(\overline{\varphi(w)}z-1\big)^{N}\Big(\log(\mu(z))+\log(\overline{\eta(w)})\Big) \\
&\qquad + Q(\overline{\varphi(w)}z) - g(z)Q(\bar{w}\sigma(z))\overline{h(w)}.
\end{align*}
Since $N$ is a nonnegative integer and $Q$ is a polynomial, the expression on the right hand side is an element of the form \eqref{Eqn:quasi_Functions}. Lemma \ref{L:finiteRankOperators} shows that $C_{\varphi}^{*} - M_{g}C_{\sigma}M^{*}_{h}$ has finite rank.
\end{proof}

\section{Adjoint formulas on $H^2(\beta)$}

In this section we would like to generalize the results in Section 2 to certain weighted Hardy spaces $H^2(\beta)$. We begin with an auxiliary result. For $s=1,2$, consider a Hilbert space $H_s$ of analytic functions on the unit disk such that
\begin{align*}
\langle z^j, z^{\ell}\rangle & = 
      \begin{cases}
         0 & \text{ if } j\neq \ell,\\
         \beta^2_s(j) & \text{ if } j=\ell.
      \end{cases}
\end{align*}
Here, $\{\beta_s(n)\}_{n=0}^{\infty}$ is a sequence of positive real numbers with $\liminf_{n\to\infty}\beta_s(n)^{1/n} = 1$. Such restriction guarantees that elements of $H_s$ are analytic function on the unit disk, see for example, \cite[Exercise 2.1.10]{CowenCRCP1995}. Assume that 
\begin{equation}
\label{Eqn:equivWeights}
\lim_{n\to\infty}\dfrac{\beta_2(n)}{\beta_1(n)}=\alpha>0.
\end{equation}
It is clear that the norms on $H_1$ and $H_2$ are equivalent. We claim that there is a compact operator $K: H_2\rightarrow H_2$ such that for all functions $f,g\in H_1$,
\begin{equation}\label{Eqn:cptPerturbation}
\alpha^2\langle f,g\rangle_{1} = \langle f,g\rangle_2 + \langle Kf,g\rangle_2.
\end{equation}
In fact, define the operator $K:H_2\rightarrow H_2$ by
\begin{align*}
K(z^n) & = \Big(\dfrac{\alpha^2\beta_1(n)^2}{\beta_2(n)^2}-1\Big)z^n,
\end{align*}
for $n=0,1,\ldots$ and extend by linearity and continuity to all $H_2$. We see that $K$ is a self-adjoint diagonal operator with respect to the orthonormal basis of monomials. By \eqref{Eqn:equivWeights}, \cite[Proposition II.4.6]{ConwaySpringer1990} shows that $K$ is a compact operator on $H_2$, hence on $H_1$ as well. It is clear that \eqref{Eqn:cptPerturbation} holds for $f(z)=z^j$ and $g(z)=z^{\ell}$ if $j\neq\ell$. If $j=\ell$, then we compute
\begin{align*}
\alpha^2\langle z^j,z^j\rangle_{1} & = \alpha^2\beta^2_1(j) = \beta_2^2(j) + \Big(\dfrac{\alpha
^2\beta_1(j)^2}{\beta_2(j)^2}-1\Big)\beta_2^2(j)\\
& = \langle z^j,z^j\rangle_2 + \langle Kz^j,z^j\rangle_2.
\end{align*}
Linearity and boundedness of $K$ then shows that \eqref{Eqn:cptPerturbation} holds for all $f,g\in H_1$.

\begin{proposition}
\label{P:compactDifference}
Let $A$ be a bounded linear operator on $H_1$ (hence, $A$ is also bounded on $H_2$). Let $B_s$ be the adjoint of $A$ on $H_s$ for $s=1,2$. Then $B_2-B_1$ is a compact operator on $H_2$ (hence, on $H_1$ as well).
\end{proposition}

\begin{proof}
For $f,g\in H_2$, we have
\begin{align*}
 \langle B_2(I+K)f,g\rangle_2 & = \langle (I+K)f,Ag\rangle_2\quad\text{(since $B_2$ is the adjoint of $A$ in $H_2$)}\\
 & = \alpha^2\langle f,Ag\rangle_1\quad\text{(by \eqref{Eqn:cptPerturbation})}\\
 & = \alpha^2\langle B_1f,g\rangle_1\quad\text{(since $B_1$ is the adjoint of $A$ in $H_1$)}\\
 & = \langle (I+K)B_1f,g\rangle_2\quad\text{(by \eqref{Eqn:cptPerturbation}).}
\end{align*}
This implies $B_2(I+K)=(I+K)B_1$, which shows that $B_2-B_1 = KB_1-B_2K$. Since $K$ is compact on $H_2$, we conclude that $B_2-B_1$ is compact as well.
\end{proof}

We now state and prove our main result in this section.

\begin{theorem}
\label{T:mainTheorem}
Let $t$ be a real number. Suppose $\beta=\{\beta(n)\}_{n=0}^{\infty}$ is a sequence of positive numbers such that
\begin{equation}
\label{Eqn:condition_main_Theorem}
\lim_{n\to\infty}\dfrac{\beta(n)}{n^{t}} = \ell,
\end{equation}
where $0<\ell<\infty$. Let $\varphi(z)=(az+b)/(cz+d)$ be a linear fractional self-map of the unit disk and $\sigma$ be its Kre\u{\i}n adjoint. Let $g(z)=(-\bar{b}z+\bar{d})^{2t-1}$ and $h(z)=(cz+d)^{-2t+1}$. Then the difference $C_{\varphi}^{*} - M_{g}C_{\sigma}M^{*}_{h}$ is a compact operator on $H^{2}(\beta)$.
\end{theorem}

\begin{proof}
Let $\alpha=-2t-1$. Then $t=-(\alpha+1)/2$ and we have
\begin{align*}
\lim_{m\to\infty}\dfrac{\beta(m)}{\|z\|_{\alpha}} & = \Big(\lim_{m\to\infty}\dfrac{\beta(m)}{m^{t}}\Big)\Big(\lim_{m\to\infty}\dfrac{m^{t}}{\|z^m\|_{\alpha}}\Big)\\
& = \ell\lim_{m\to\infty}\dfrac{m^{-(\alpha+1)/2}}{\|z^m\|_{\alpha}}.
\end{align*}
The last limit is a finite positive number by Remark \ref{R:asymptoticBehavior}. This, in particular, says that the spaces $A^2_{\alpha}$ and $H^2(\beta)$ are the same with equivalent norms. For any bounded operator $T$ on these spaces, we write $T^{*,\alpha}$ for the adjoint of $T$ as an operator on $A^2_{\alpha}$ and $T^{*,\beta}$ for the adjoint of $T$ as an operator on $H^2(\beta)$.

By Theorems \ref{T:AdjointFormula_a} and \ref{T:AdjointFormula_b}, the difference $K=C_{\varphi}^{*,\alpha} - M_{g}C_{\sigma}M_{h}^{*,\alpha}$ is compact on $A^2_{\alpha}$, hence on $H^2(\beta)$ as well.

On the other hand, applying Proposition \ref{P:compactDifference} with $H_1=A^2_{\alpha}$ and $H_2=H^2(\beta)$, we have $C_{\varphi}^{*,\beta} = C_{\varphi}^{*,\alpha}+K_1$ and $M_{h}^{*,\beta} = M_{h}^{*,\alpha}+K_2$ for some compact operators $K_1, K_2$ on $H^2(\beta)$. Consequently,
\begin{align*}
C_{\varphi}^{*,\beta} - M_{g}C_{\sigma}M_h^{*,\beta} & = (C_{\varphi}^{*,\alpha}+K_1) - M_{g}C_{\sigma}(M_h^{*,\alpha}+K_2)\\
& = C_{\varphi}^{*,\alpha} - M_{g}C_{\sigma}M_h^{*,\alpha} + K_1 - M_{g}C_{\sigma}K_2\\
& = K + K_1 - M_gC_{\sigma}K_2,
\end{align*}
which is compact on $H^2(\beta)$. This completes the proof of the theorem.
\end{proof}

We now explain how one obtains Heller's results from our Theorem \ref{T:mainTheorem}. Let $\varphi$ be a holomorphic self-map of the unit disk. We shall consider two particular cases: the case $\varphi(0)=0$ and the case $\varphi$ is an automorphism.

\begin{corollary}
\label{C:improved_TheoremA}
Let $\beta=\{\beta(n)\}_{n=0}^{\infty}$ be a sequence of positive numbers satisfying the condition \eqref{Eqn:condition_main_Theorem}. Let $\varphi(z)=az/(cz+d)$ be a holomorphic self-map of the disk and consider $C_{\varphi}$ acting on $H^2(\beta)$. Then we have
$$C_{\varphi}^{*} = M_{G}^{*}C_{\sigma} \, \text{ mod } \, \mathcal{K},$$
where $G(z)=\big(-(c/a)z+1\big)^{2t-1}$ and $\sigma(z)=(\bar{a}/\bar{d})z-\bar{c}/\bar{d}$ is the Kre\u{\i}n adjoint of $\varphi$.
\end{corollary}

\begin{proof}
Theorem \ref{T:mainTheorem} shows that
\begin{align}
\label{Eqn:adjointCOs}
C_{\varphi}^{*} = M_{g}C_{\sigma}M_h^{*} \ \text{ mod } \ \mathcal{K},
\end{align}
where $g(z)=(\bar{d})^{2t-1}$,\, $h(z)=(cz+d)^{-2t+1}$. Since $g$ is a constant function, we may combine it with $h$ and rewrite \eqref{Eqn:adjointCOs} as $C_{\varphi}^{*} = C_{\sigma}M^{*}_{h_1}\, \text{ mod }\, \mathcal{K}$,
where $h_1(z)=\big(d/(cz+d)\big)^{2t-1}$. It then follows that $C_{\sigma} = C_{\varphi}^{*}M_{1/h_1}^{*} \, \text{ mod } \, \mathcal{K}$. Now, a direct calculation verifies that $h_1=G\circ\varphi$. We then compute
\begin{align*}
C_{\sigma} = C_{\varphi}^{*}M_{1/G\circ\varphi}^{*} = \Big(M_{1/G\circ\varphi}C_{\varphi}\Big)^{*} = \Big(C_{\varphi}M_{1/G}\Big)^{*} = M^{*}_{1/G}C_{\varphi}^{*}.
\end{align*}
Multiplying by $M_{G}^{*}$ on the left gives $C_{\varphi}^{*} = M_{G}^{*}C_{\sigma} \, \text{ mod } \, \mathcal{K}$ as desired.
\end{proof}

\begin{corollary}
\label{C:improved_TheoremB}
Let $\beta=\{\beta(n)\}_{n=0}^{\infty}$ be a sequence of positive numbers satisfying the condition \eqref{Eqn:condition_main_Theorem}. Let $\varphi(z)=\lambda(z+u)/(1+\bar{u}z)$, $|\lambda|=1$, $|u|<1$, be an automorphism of the disk and consider $C_{\varphi}$ acting on $H^2(\beta)$. Then $$C_{\varphi}^{*} = M_{G}^{*}C_{\varphi^{-1}}M_{1/H} \text{ mod } \mathcal{K},$$
where $G(z)=(-\overline{\lambda u}\,z+1)^{2t-1}$ and $H(z)=(\bar{u}z+1)^{2t-1}$.
\end{corollary}

\begin{proof} It can be verified that $\sigma = \varphi^{-1}$. Theorem \ref{T:mainTheorem} gives
$$C_{\varphi}^{*} = M_{g}C_{\varphi^{-1}}M_h^{*} \ \text{ mod } \ \mathcal{K},$$
where $g(z)=(-\overline{\lambda u}z+1)^{2t-1}$ and $h(z)=(\bar{u}z+1)^{-2t+1}$. Taking adjoints gives
\begin{align*}
C_{\varphi} & = \Big(M_{g}C_{\varphi^{-1}}M^{*}_{h}\Big)^{*}\ \text{ mod } \ \mathcal{K}\\
& = M_{h}C_{\varphi^{-1}}^{*}M_{g}^{*}\ \text{ mod } \ \mathcal{K},
\end{align*}
which implies
\begin{align*}
M_{1/h}C_{\varphi}M^{*}_{1/g} = C_{\varphi^{-1}}^{*} \  \text{ mod } \ \mathcal{K}.
\end{align*}
Taking inverses then yields
\begin{align*}
C_{\varphi}^{*} = \big(C_{\varphi^{-1}}^{*}\big)^{-1} = \Big(M_{1/h}C_{\varphi}M^{*}_{1/g}\Big)^{-1} = M^{*}_{g}C_{\varphi^{-1}}M_{h} \ \text{ mod } \ \mathcal{K}.
\end{align*}
Since $g=G$ and $h=1/H$, the conclusion of the corollary follows.
\end{proof}

The space $S^2(\D)$ can be identified as $H^2(\beta)$, where the weight sequence  $\beta=\{\beta(n)\}_{n\geq 0}$ is given by $\beta(0)=1$ and $\beta(n)=n$ for all $n\geq 1$. This sequence satisfies condition \eqref{Eqn:condition_main_Theorem} with $t=1$. Consequently, Theorem A follows from Corollary \ref{C:improved_TheoremA} and Theorem B follows from Corollary \ref{C:improved_TheoremB}.

\subsection*{Acknowledgements} The authors wish to thank the referee for a careful reading and useful comments that improved the presentation of the paper.

\end{document}